\documentclass[12pt,oneside,reqno]{amsart}
\usepackage{amsmath,amssymb}
\usepackage{amscd}

\theoremstyle{plain}
\newtheorem{thm}{Theorem}[section]
\newtheorem{theorem}[thm]{Theorem}
\newtheorem*{mainthm}{Main Theorem}

\newtheorem{lemma}[thm]{Lemma}

\newtheorem{proposition}[thm]{Proposition}
\theoremstyle{definition}

\newtheorem{remark}[thm]{Remark}

\numberwithin{equation}{section}

\newcommand{\C}{{\mathbb C}}

\newcommand{\Q}{{\mathbb Q}}
\newcommand{\R}{{\mathbb R}}

\newcommand{\Z}{{\mathbb Z}}

\newcommand{\id}{{\rm id}}
\newcommand{\tr}{{\rm trace}}

\newcommand{\Aut}{{\operatorname{Aut}}}
\newcommand{\ord}{{\operatorname{ord}}}
\def\rank{\operatorname{rank}}
\newcommand{\disc}{{\operatorname{disc}}}

\newcommand{\Lprime}{\mathbb{L}^{\star}_{(1,1)}}

\title[K3 surfaces and Pell equation]{K3 surfaces with Picard number 2, Salem Polynomials and Pell equation}
\author{Kenji Hashimoto, JongHae Keum and Kwangwoo Lee}
\date{\today}
\begin{document}

\begin{abstract}
If an automorphism of a projective K3 surface
 with Picard number $2$ is of infinite order, then the automorphism
 corresponds to a solution of Pell equation.
In this paper, by solving this equation, we determine all Salem polynomials of symplectic and anti-symplectic automorphisms of projective K3 surfaces with Picard number $2$.
\end{abstract}

\maketitle  \setcounter{tocdepth}{1}

\section{Introduction}
A compact complex surface $X$
 is called a K3 surface if it is simply connected and
 has a nowhere vanishing holomorphic $2$-from $\omega_X$.
 By the global Torelli theorem the automorphism group of a K3 surface is determined, up to subquotient of finite index, by its Picard lattice. Suppose that a K3 surface $X$ is projective and has Picard number $2$.  Galluzzi, Lombardo and Peters \cite{GLP}
 applied the classical theory of binary quadratic forms (cf.\ \cite{D,J}) to prove that the automorphism group Aut$(X)$ is finite cyclic, infinite cyclic or infinite dihedral.

Suppose that
 $\varphi \in \Aut(X)$ has infinite order and is either symplectic or anti-symplectic,
 where we have $\varphi^* \omega_X =\varepsilon\omega_X$ and
 $\varphi$ is said to be symplectic (resp.\ anti-symplectic)
 if $\varepsilon=1$ (resp. $\varepsilon=-1$).
In the present paper,
 we determine all Salem polynomials of such automorphisms. The Salem polynomial of $\varphi$
 is defined as the characteristic polynomial of $\varphi^*|_{S_X}$ in this case.
Here $S_X$ denotes the Picard lattice of $X$.
Since $\det(\varphi^*|_{S_X})=1$,
 it is enough to determine $\tr(\varphi^*|_{S_X})$, and
 the Salem polynomial of $\varphi$ is given as
\begin{equation}
 x^2-t x+1, \quad t=\tr(\varphi^*|_{S_X}).
\end{equation}
We construct an automorphism $\varphi$ whose trace on $S_X$ is as in the condition (2) in Main Theorem below.
We use a result on class numbers of real quadratic fields for the construction.

We apply the theory of binary quadratic forms to study actions on $S_X$ (cf.\ \cite{GLP}).
The (induced) action $\varphi^*|_{S_X}$ of
 $\varphi$ on $S_X$ as above
 corresponds to an integer solution $(u,v)$
 to the Pell equation
\begin{equation}\label{pell}
 u^2-D v^2=4, \quad D:=-\disc(S_X) >0.
\end{equation}
More precisely, let
\begin{equation}
 \begin{pmatrix}2a&b\\b&2c\end{pmatrix}, \quad
 a,b,c \in \Z,
\end{equation}
 be the intersection matrix of $S_X$
 (under some basis).
Then $D=b^2-4ac$ and $\varphi^*|_{S_X}$
 is represented as the matrix
\begin{equation}
 g_{u,v}=
 \begin{pmatrix}
  (u-bv)/2 & -cv \\
  av & (u+bv)/2
 \end{pmatrix},
\end{equation}
 where $(u,v)$ is an integer solution to the Pell equation (\ref{pell})
 with $u>0$ (see Proposition \ref{PROP_generator_of_isometry}).
Moreover, we show
\begin{equation}
 (u,v)=(\alpha^2-2 \varepsilon, \alpha \beta), \quad
 \alpha^2-D \beta^2=4 \varepsilon
\end{equation}
 for some integers $\alpha,\beta$
 (see Proposition \ref{PROP_key_prop}).
In particular, we have $g_{u,v}=g_{\alpha,\beta}^2$.
%
Our goal is to determine all possible values of
\begin{equation}
 \tr(\varphi^*|_{S_X})=u=\alpha^2-2 \varepsilon.
\end{equation}
The result is stated as follows:

\begin{mainthm} \label{THM_main}
For $\varepsilon\in \{ \pm 1 \}$ and $u\in\Z$,
 the following conditions are equivalent:
\begin{enumerate}
\item
$u=\tr(\varphi^*|_{S_X})$
 for some automorphism $\varphi$
 of a projective K3 surface $X$ with Picard number $2$ such that
 $\ord(\varphi)=\infty$ and
 $\varphi^* \omega_X=\varepsilon \omega_X$.
\item
$u=\alpha^2-2 \varepsilon$ for some $\alpha \in A_\varepsilon$,
 where
\begin{equation}
 A_\varepsilon=
 \begin{cases}
  \Z_{\geq 4} & \text{if\ \ } \varepsilon=1, \\
  \Z_{\geq 4} \setminus \{ 5,7,13,17 \} & \text{if\ \ } \varepsilon=-1.
 \end{cases}
\end{equation}
\end{enumerate}
\end{mainthm}


%
\begin{remark}
%
(i)
For a given trace $u$ as in Main Theorem,
we have only finitely many pairs $(X, \varphi)$ with $\tr(\varphi^*|_{S_X})=u$, up to equivariant deformation. Indeed, for a given $u$
there are only finitely many $D$'s satisfying the Pell equation ($\ref{pell}$) for some integer $v$, and for a given $D$ there are only finitely many lattices of rank $2$ with discriminant $-D$ up to isomorphism (see \cite[Chap.\ 15]{CS}). Each of such lattices corresponds to a connected moduli space (of dimension $20-\rank S_X=18$) of
 $(X,\varphi)$ such that $S_X$ is isomorphic to the given lattice.
(ii)
An automorphism of infinite order of a projective K3 surface $X$ with Picard number $2$ always acts on the Picard lattice $S_X$ with determinant $1$.
\end{remark}
As an application, we show that every projective K3 surface with Picard number $2$ admitting a fixed-point-free automorphism must have the same Picard lattice as the K3 surfaces with the Cayley--Oguiso automorphism \cite{O,Cayley-Oguiso-auto}, and that the action of the automorphism on the Picard lattice is the same as that of the Cayley--Oguiso automorphism
 (see Section \ref{SECT_app}).

The structure of this paper is as follows:
We recall some results
 on lattices and their isometries in Section \ref{SECT_lattices},
 and K3 surfaces and their automorphisms in Section \ref{SECT_k3}.
We prove Main Theorem in Section \ref{SECT_proof}.
Finally, in Section \ref{SECT_app},
 a few applications of our result are discussed.

\subsection*{Acknowledgment}
This work was supported by National Research Foundation of Korea (NRF-2007-0093858).
The first author is partially supported by JSPS Grant-in-Aid for Young Scientists (B), No.\ 17K14156.
We would like to thank Professor Shigeru Mukai for useful discussion.

\section{Lattices} \label{SECT_lattices}


A lattice
 is a free $\Z$-module $L$
 of finite rank equipped with a symmetric bilinear form
 $\langle ~,~ \rangle \colon L\times L\rightarrow \Z$.
If $x^2:=\langle x,x\rangle \in 2\Z$ for any $x\in L$,
 a lattice $L$ is said to be even.
We fix a $\Z$-basis of $L$ and identify the lattice $L$
 with its Gram matrix $Q_L$ under this basis.
The discriminant $\disc(L)$ of $L$ is defined as $\det(Q_L)$,
 which is independent of the choice of basis.
A lattice $L$ is called non-degenerate if
 $\disc(L)\neq 0$ and unimodular if $\disc(L)=\pm1$.
For a non-degenerate lattice $L$,
 the signature of $L$ is defined as $(s_+,s_-)$,
 where $s_{+}$ (resp.\ $s_-$)
 denotes the number of the positive (resp.\ negative) eigenvalues of
 $Q_L$.
An isometry of $L$ is an isomorphism of the $\Z$-module $L$
 preserving the bilinear form.
The orthogonal group $O(L)$ of $L$ consists of the isometries of $L$.
We consider $L$ as $\Z^n$ (column vectors) with Gram matrix $Q_L$
 and use the following identification:
\begin{equation} \label{EQ_ortho}
 O(L)= \{ g\in GL_n(\Z) \bigm| g^T \cdot Q_L \cdot g=Q_L \}, \quad
 n=\rank L.
\end{equation}
For a non-degenerate lattice $L$,
 the discriminant group $A(L)$ of $L$ is defined by
\begin{equation} 
 A(L):=L^*/L, \quad
 L^*:=
 \{ x\in L\otimes \Q \bigm|
 \langle x,y \rangle \in \Z ~ (\forall y\in L) \}.
\end{equation}
%
%
%
%
Let $K$ be a sublattice of a lattice $L$,
 that is, $K$ is a $\Z$-submodule of $L$
 equipped with the restriction of the bilinear form of $L$ to $K$.
If $L/K$ is torsion-free as a $\Z$-module,
 $K$ is said to be primitive.
%
%
%
For a non-degenerate lattice $L$ of signature $(1,k)$ with $k\geq 0$,
 we have the decomposition
\begin{equation} \label{EQ_positive_cone}
 \{x\in L\otimes \R \bigm| x^2 >0 \}=C_L\sqcup (-C_L)
\end{equation}
 into two disjoint cones.
Here $C_L$ and $-C_L$ are connected components.
We define
\begin{equation} \label{O^+}
 O^+(L):=\{g\in O(L) \bigm| g(C_L)=C_L \}, \quad
 SO^+(L):=O^+(L)\cap SO(L),
\end{equation}
 where $SO(L)$ is the subgroup of $O(L)$
 consisting of isometries of determinant $1$.
The group $O^+(L)$ is a subgroup of $O(L)$ of index $2$.

\begin{lemma} \label{LEM_action_disc_grp}
Let $L$ be a non-degenerate even lattice of rank $n$.
For $g\in O(L)$ and $\varepsilon \in \{ \pm 1 \}$,
 $g$ acts on $A(L)$ as $\varepsilon \cdot \id$
 if and only if $(g- \varepsilon \cdot I_n) \cdot Q_{L}^{-1}$
 is an integer matrix.
\end{lemma}
\begin{proof}
As in (\ref{EQ_ortho}),
 we consider $L$ as $\Z^n$ with Gram matrix $Q_L$.
Then $L^*$ is generated by
 the columns of $Q_L^{-1}$.
This implies the assertion.
\end{proof}


%

In the rest of this section, we recall a few results related to lattices of signature $(1,1)$, or indefinite binary quadratic forms.

\begin{proposition} (\cite{D,J}) \label{PROP_generator_of_isometry}
Let $L$ be an even lattice of signature $(1,1)$:
\begin{equation}
 L:= \begin{pmatrix} 2a& b\\ b&2c \end{pmatrix}, \quad
 D:=-\disc(L)=b^2-4 ac>0.
\end{equation}
Then $SO^+(L)$ consists of the elements of the form
\begin{equation} \label{EQ_u}
 g=
 \begin{pmatrix}
  (u-bv)/2 & -cv \\
  av & (u+bv)/2
 \end{pmatrix}.
\end{equation}
Here $(u,v)$ is any solution of the positive Pell equation
\begin{equation} \label{EQ_Pell}
 u^2-D v^2=4
\end{equation}
 with $u,kv\in\Z$ and $u>0$, where $k:=\operatorname{gcd}(a,b,c)$.
If $D$ is a square number (resp.\ not a square number),
 then $SO^+(L)$ is isomorphic to a trivial group (resp.\ $\Z$).
\end{proposition}

\begin{remark}
In Proposition \ref{PROP_generator_of_isometry},
 $D$ is a square number if and only if
 there exists $v \in L$ with $v^2=0$.
\end{remark}

\begin{remark}
Pell equation is usually considered as a Diophantine equation,
 that is, only integer solutions are admitted.
However, in Proposition \ref{PROP_generator_of_isometry},
 we need non-integer solutions to find all elements in $SO^+(L)$.
Of course we have usual Pell equation by replacing
 $a,b,c$ and $D$ by $a/k,b/k,c/k$ and $D/k^2$.
\end{remark}

\begin{remark} \label{RMK_minimal_solution}
In Proposition \ref{PROP_generator_of_isometry},
 the eigenvalues of $g$ 
 are $(u\pm v \sqrt{D})/2$ and
 the solution $(u',v')$ of  (\ref{EQ_Pell}) corresponding to
 $g^k$ ($k\in\Z$) satisfies
\begin{equation}
 \frac{u'+v' \sqrt{D}}{2}=\left( \frac{u+v \sqrt{D}}{2} \right)^k.
\end{equation}
Let $(u_0,v_0)$ be the smallest positive solution of (\ref{EQ_Pell}),
 that is, $(u_0,v_0)$ is the solution with $u_0,v_0>0$
 which minimizes $u_0$ (or $v_0$).
Then $SO^+(L)$ is generated by $g_0$ corresponding to $(u_0,v_0)$.

\end{remark}



We define
\begin{align}
 \mathbb{L}_{(1,1)} &:= \label{EQ_L11}
 \text{the set of even lattices of rank $2$ and signature $(1,1)$}, \\
 \Lprime
 &:= \{ L\in \mathbb{L}_{(1,1)} \bigm|
 v^2 \not \in \{ 0,-2 \} ~ (\forall v\in L) \}. \label{EQ_LL11}
\end{align}

\begin{lemma} \label{LEM_class_number}
Let $K=\Q(\sqrt{D})$
 be a real quadratic field of discriminant $D$ ($>0$).
Then there exists $L\in \Lprime$ with $\disc(L)=-D$
 if and only if $h^+(D)>1$.
Here $h^+(D)$ denotes the narrow class number of $K$.
\end{lemma}
\begin{proof}
It is classically known that there is a natural bijection between
 the set of isomorphism classes of
 $L \in \mathbb{L}_{(1,1)}$ with $\disc(L)=-D$
 and the narrow ideal class group of $K$
 (see e.g.\ \cite[Section VII.2, Theorem 58]{FT-number-theory}).

Let $L\in \mathbb{L}_{(1,1)}$.
There exists $v\in L$ with $v^2=0$ if and only if
 $-\disc(L)$ is a square number.
Hence, for any $L$ such that
 $\disc(L)=-D$ and $L \not\in \Lprime$,
 there exists $v\in L$ with $v^2=-2$ and
\begin{equation} \label{EQ_lat_-2}
 L\cong L_0:=
 \begin{pmatrix} -2 & \delta \\ \delta & (D-\delta^2)/2 \end{pmatrix},
\end{equation}
 where $\delta\in\{ 0,1 \}$ and $D\equiv \delta \pmod 2$.
Therefore the existence of
 $L\in \Lprime$ with $\disc(L)=-D$
 is equivalent to the existence of
 $L\in \mathbb{L}_{(1,1)}$ with $\disc(L)=-D$
 and $L\not\cong L_0$,
 which is also equivalent to the condition $h^+(D)>1$.
\end{proof}


\section{K3 surfaces}
\label{SECT_k3}


A compact complex surface $X$
 is called a K3 surface if it is simply connected and
 has a nowhere vanishing holomorphic $2$-from $\omega_X$
 (see \cite{BHPV} for details).
We consider the second integral cohomology $H^2(X,\Z)$
 with the cup product as a lattice.
It is known that $H^2(X,\Z)$ is an even unimodular lattice
 of signature $(3,19)$,
 which is unique up to isomorphism and is called the K3 lattice.
We fix such a lattice and denote it by $\Lambda_{\text{K3}}$.
The Picard lattice $S_X$
 and transcendental lattice $T_X$ of $X$ are defined by
\begin{align}
 S_X &:= \{ x \in H^2(X,\Z) \bigm| \langle x,\omega_X \rangle=0 \}, \\
 T_X &:= \{ x \in H^2(X,\Z) \bigm|
 \langle x,y \rangle=0~(\forall y \in S_X) \}.
\end{align}
Here $\omega_X$ is considered as an element in $H^2(X,\C)$
 and the bilinear form on $H^2(X,\Z)$ is extended to that on
 $H^2(X,\C)$ linearly.
The Picard group 
 of $X$
 is naturally isomorphic to $S_X$.
It is known that $X$ is projective if and only if
 $S_X$ has signature $(1,\rho-1)$,
 where $\rho=\rank S_X$ is the Picard number of $X$.


Let $X$ be a projective K3 surface.
Since $H^2(X,\Z)$ is unimodular and $S_X$ is non-degenerate,
 we have the following natural identification:
\begin{equation}
 A(S_X)=A(T_X)=H^2(X,\Z)/(S_X\oplus T_X)
\end{equation}
 (see \cite{nikulin79int} for details).
By the global Torelli theorem for K3 surfaces \cite{PS,burnsrapoport75},
 the following map is injective:
\begin{equation} \label{EQ_map_torelli}
 \Aut(X) \ni \varphi \mapsto
 (g,h):=(\varphi^*|_{S_X},\varphi^*|_{T_X})
 \in O(S_X) \times O(T_X).
\end{equation}
Moreover,
 $(g,h) \in O(S_X) \times O(T_X)$
 is the image of some $\varphi \in \Aut(X)$
 by the map (\ref{EQ_map_torelli}) if and only if
 (1) the linear extension of $g$ (resp.\ $h$)
 preserves the ample cone $C_X$ of $X$ (resp.\ $\C\omega_X$)
 and (2) the actions of $g$ and $h$ on $A(S_X)=A(T_X)$
 coincide.

In this paper,
 we study projective K3 surfaces $X$ with Picard number $2$,
 for which we have $S_X \in \mathbb{L}_{(1,1)}$
 (see (\ref{EQ_L11})).
Conversely, the following holds:

\begin{proposition} \label{PROP_pic_2_sx}
For any $L \in \mathbb{L}_{(1,1)}$,
 there exists a (projective) K3 surface $X$
 such that $S_X \cong L$.
\end{proposition}
\begin{proof}
Let $L \in \mathbb{L}_{(1,1)}$.
The K3 lattice $\Lambda_{\text{K3}}$
 contains a primitive sublattice $S$ which is isomorphic to $L$
 \cite[Theorem 1.14.4]{nikulin79fin}.
(In fact, such an $S$ is unique up to $O(\Lambda_{\text{K3}})$.)
The surjectivity of the period map for K3 surfaces \cite{Todorov}
 implies that there exists a K3 surface $X$
 such that $S_X \cong L$.
\end{proof}

Let $X$ be a projective K3 surface with Picard number $2$.
%
%
%
%
We are interested in automorphisms $\varphi \in \Aut(X)$
 of infinite order such that
 $\varphi^*\omega_X=\varepsilon \omega_X$
 with $\varepsilon \in \{ \pm 1 \}$.
By applying the global Torelli theorem, we obtain the following:


\begin{proposition} \label{PROP_inf_S_X}
Let $X$ be a projective K3 surface with Picard number $2$
 and $(g,h) \in O(S_X) \times O(T_X)$.
Then $(g,h)=(\varphi^*|_{S_X},\varphi^*|_{T_X})$ for some
 $\varphi \in \Aut(X)$ such that $\ord(\varphi)=\infty$
 and
 $\varphi^* \omega_X=\varepsilon \omega_X$
 with $\varepsilon \in \{ \pm 1 \}$
 if and only if the following conditions are satisfied:
\begin{enumerate}
\item
$S_X \in \Lprime$;
\item
$1\neq g \in SO^+(S_X)$;
\item
$g$ acts on $A(S_X)$ as $\varepsilon \cdot \id$;
\item
$h=\varepsilon \cdot \id$.
\end{enumerate}
\end{proposition}
\begin{proof}
Assume that $(g,h)=(\varphi^*|_{S_X},\varphi^*|_{T_X})$ for some $\varphi$
 as in the statement.
Since $\Aut(X)$ is infinite if and only if
 $S_X \in \Lprime$ \cite[Corollary 3.4]{GLP},
 the condition (1) holds.
Moreover, the ample cone $C_X$ of $X$ coincides with
 $C_{S_X}$ or $-C_{S_X}$ (see (\ref{EQ_positive_cone})).
Since $\varphi^* \omega_X=\varepsilon \omega_X$,
 $h$ acts
 on $T_X$ as $\varepsilon \cdot \id$
 \cite[Theorem 3.1]{nikulin79fin}.
Thus the condition (4) holds.
This implies the condition (3)
 because the actions of $g$ and $h$ on $A(S_X)=A(T_X)$ coincide.
If $\det(g)=-1$, then the action of $g$ on $C_X$ is a reflection, and hence
 $g^2=1$.
Since $h^2=1$, we have $\varphi^2=1$
 by the injectivity of the map (\ref{EQ_map_torelli}),
 which is a contradiction.
Therefore $g\in SO^+(S_X)$.
Similarly, we have $g \neq 1$.
Hence the condition (2) holds.

Conversely, if the conditions (1)--(4) are satisfied,
 then there exists $\varphi \in \Aut(X)$ such that
 $(g,h)=(\varphi^*|_{S_X},\varphi^*|_{T_X})$
 by the global Torelli theorem.
The condition (4) implies that
 $\varphi^* \omega_X=\varepsilon \omega_X$
 because $\omega_X \in T_X \otimes \C$.
Since $SO^+(S_X)$ is isomorphic to $\Z$
 (Proposition \ref{PROP_generator_of_isometry}),
 $\varphi$ is of infinite order.
\end{proof}

\begin{remark}
A non-projective K3 surface with Picard number $2$ may have
 an automorphism of infinite order (cf.\ \cite{mcmullen1}).
\end{remark}


\section{Proof of Main Theorem}
\label{SECT_proof}


\subsection{Preparation}

In order to prove Main Theorem,
 we show the following:

\begin{proposition} \label{PROP_key_prop}
Let $X$ be a projective K3 surface with Picard number $2$.
For $\varepsilon\in \{ \pm 1 \}$ and $g \in O(S_X)$,
 the following conditions are equivalent:
\begin{enumerate}
\item
$g=\varphi^*|_{S_X}$ for some $\varphi \in \Aut(X)$
 such that $\ord(\varphi)=\infty$ and
 $\varphi^* \omega_X=\varepsilon \omega_X$.
\item
$S_X \in \Lprime$ and
 $g$ is given by
\begin{equation}
 g=g_{u,v}:=\begin{pmatrix}
  (u-bv)/2 & -cv \\
  av & (u+bv)/2
 \end{pmatrix},
\end{equation}
 where
\begin{equation}
 S_X=\begin{pmatrix}2a&b\\b&2c\end{pmatrix}, \quad
 (u,v)=(\alpha^2-2\varepsilon,\alpha\beta),
\end{equation}
 and $\alpha,\beta$ are nonzero integers satisfying
\begin{equation}
 \alpha^2-D \beta^2=4\varepsilon, \quad
 D:=-\disc(S_X)=b^2-4ac>0.
\end{equation}
\end{enumerate}
\end{proposition}

\begin{remark}
Assume that $D$ is not a square number and
 $\operatorname{gcd}(a,b,c)=1$
 in Proposition \ref{PROP_key_prop}.
Then Proposition \ref{PROP_generator_of_isometry}
 implies that
\begin{equation}
 SO^+(S_X)= \{ g_{u,v} \bigm| u,v\in\Z, ~
 u>0, ~ u^2-D v^2=4 \} \cong \Z.
\end{equation}
Under the condition (2) in Proposition \ref{PROP_key_prop},
 we have
\begin{equation}
 \left( \frac{\alpha+\beta\sqrt{D}}{2} \right)^2=
 \frac{\alpha^2-2\varepsilon+\alpha\beta\sqrt{D}}{2}
 =\frac{u+v\sqrt{D}}{2}.
\end{equation}
Hence $g_{u,v}=g_{\alpha,\beta}^2$
 (see Remark \ref{RMK_minimal_solution}),
 and thus
\begin{align}
 \Gamma:=& \{ g\in SO^+(S_X) \bigm| g=\varphi^* |_{S_X},~
 \varphi^* \omega_X= \pm \omega_X ~
 (\exists \varphi \in \Aut(X)) \} \\
 =& \{ g_{\alpha,\beta}^2 \bigm|
 \alpha,\beta \in \Z,~
 \alpha^2-D \beta^2=\pm 4 \}.
\end{align}
Therefore $\Gamma$ is a subgroup of $SO^+(S_X)$
 of index $1$ or $2$
 according to whether the equation
 $\alpha^2-D\beta^2=-4$ has an integer solution
 $(\alpha,\beta)$ or not.
\end{remark}

We obtain Proposition \ref{PROP_trace_g} below from
 Proposition \ref{PROP_key_prop},
 because any lattice
 $L\in \mathbb{L}_{(1,1)}$
 is realized as the Picard lattice $S_X$
 of a K3 surface $X$
 by Proposition \ref{PROP_pic_2_sx}.


\begin{proposition} \label{PROP_trace_g}
For $\varepsilon\in \{ \pm 1 \}$ and $u\in\Z$,
 the following conditions are equivalent:
\begin{enumerate}
\item
$u=\tr(\varphi^*|_{S_X})$
 for some automorphism $\varphi$
 of a projective K3 surface $X$ with Picard number $2$ such that
 $\ord(\varphi)=\infty$ and
 $\varphi^* \omega_X=\varepsilon \omega_X$.
\item
$u=\alpha^2-2 \varepsilon$ for some $\alpha \in \Z_{>0}$ such that
 there exist $L\in \Lprime$ and $\beta \in \Z_{>0}$
 satisfying 
 $\disc(L)=-(\alpha^2-4 \varepsilon)/\beta^2$.
\end{enumerate}
\end{proposition}


Proposition \ref{PROP_key_prop}
 is a direct conclusion of Proposition \ref{PROP_inf_S_X}
 and the following:

\begin{lemma} \label{LEM_disc_triv}
Let $L\in\mathbb{L}_{(1,1)}$:
\begin{equation}
 L=\begin{pmatrix}2a&b\\b&2c\end{pmatrix},
 \quad D:=-\disc(L)=b^2-4ac>0,
\end{equation}
and let $g$ be a nontrivial isometry contained in $SO^+(L)$
 (see Proposition \ref{PROP_generator_of_isometry}):
\begin{equation} \label{EQ_g_and_pell}
 g=\begin{pmatrix}
  (u-bv)/2 & -cv \\
  av & (u+bv)/2
 \end{pmatrix},
 \quad u^2-D v^2=4, \quad u>2.
\end{equation}
For $\varepsilon \in \{ \pm 1 \}$,
 the isometry $g$ acts on $A(L)$ as $\varepsilon \cdot \id$
 if and only if we have
\begin{equation} \label{EQ_D_u_v}
 (u,v)=(\alpha^2-2\varepsilon,\alpha\beta), \quad
 \alpha^2-D \beta^2=4\varepsilon
\end{equation}
 for some nonzero integers $\alpha,\beta$.
\end{lemma}
\begin{proof}
Recall that we have $u,kv \in \Z$ and $u>0$,
 where $k:=\operatorname{gcd}(a,b,c)$.
Moreover, since $g\neq 1$, it follows that $u>2$.
We have
\begin{equation}
 \det(g-\varepsilon \cdot I_2)
 =\det(g)-\varepsilon \cdot \tr(g)+1
 =-\varepsilon u+2
\end{equation}
 and
 \begin{align}
  (g-\varepsilon \cdot I_2) \cdot Q_L^{-1} &= \frac{-1}{D}
  \begin{pmatrix}
   (u-bv)/2-\varepsilon & -cv \\
   av & (u+bv)/2-\varepsilon
  \end{pmatrix}
  \begin{pmatrix}2c&-b\\-b&2a\end{pmatrix} \\
  &= \frac{-1}{D}
  \begin{pmatrix}
   c(u-2\varepsilon)  & (Dv-b(u-2\varepsilon))/2 \\
   -(Dv+b(u-2\varepsilon))/2 & a(u-2\varepsilon) \label{EQ_mat_g_epsilon}
  \end{pmatrix}.
 \end{align}

Assume that $g$ acts on $A(L)$ as $\varepsilon \cdot \id$.
By Lemma \ref{LEM_action_disc_grp},
 $(g-\varepsilon \cdot I_2) \cdot Q_L^{-1}$ is an integer matrix,
 and hence
 $D$ divides $\det(g- \varepsilon \cdot I_2)=-\varepsilon u+2$.
Thus $u-2\varepsilon=mD$ for some $m\in \Z_{>0}$.
Therefore
\begin{equation} \label{EQ_v_square}
 v^2=\frac{1}{D}(u^2-4)=m(mD+4 \varepsilon).
\end{equation}
In particular, we have $v\in \Z$.

{\bf Claim.}
$m$ is a square number.

This is shown by applying (\ref{EQ_v_square}) as follows.
Suppose that $p$ is a prime number and $p^e \mid\mid m$
 with $e\in \Z_{>0}$
 (that is, $p^e \mid m$ and $p \nmid (m/p^e)$).
It is enough to show that $e$ is always even.
(i) If $p$ is odd, then $p^e \mid\mid m(mD+4 \varepsilon)=v^2$.
Hence $e$ is even.
(ii) In the case $p=2$, if $e\geq 3$,
 then $mD+4 \varepsilon \equiv 4 \pmod 8$.
Hence $2^{e+2} \mid\mid m(mD+4\varepsilon)=v^2$,
 and thus $e$ is even.
If $e=1$, then
 $(v/2)^2=(m/2)^2 D+\varepsilon m \equiv 2$ or $3 \pmod 4$
 because $D=b^2-4ac\equiv 0$ or $1 \pmod 4$.
On the other hand, $v$ is even and
 $(v/2)^2\equiv 0$ or $1 \pmod 4$,
 which is a contradiction.
This completes the proof of Claim.

By Claim, we have $m=n^2$ for some $n\in \Z_{>0}$ and
\begin{equation}
 u-2\varepsilon=n^2D.
\end{equation}
By (\ref{EQ_v_square}),
 we have $v^2=n^2(n^2 D+4\varepsilon)$.
Thus $v=\alpha\beta$ and
 $n^2D+4\varepsilon=\alpha^2$ for some $\alpha \in \Z_{>0}$
 and $\beta\in \{ \pm n \}$.
This implies (\ref{EQ_D_u_v}).

Conversely, assume that $D$ and $(u,v)$ satisfy (\ref{EQ_D_u_v})
 for some nonzero integers $\alpha,\beta$.
In order to show that
 $g$ acts on $A(L)$ as $\varepsilon \cdot \id$,
 it is enough to check that the matrix (\ref{EQ_mat_g_epsilon})
 is an integer matrix
 by Lemma \ref{LEM_action_disc_grp}.
By (\ref{EQ_D_u_v}), we have
\begin{equation}
 \frac{u-2\varepsilon}{D}=\beta^2, \quad
 \frac{Dv \pm b(u-2\varepsilon)}{2D}
 =\frac{v \pm b \beta^2}{2}.
\end{equation}
If $v \pm b \beta^2$ is odd, then
\begin{equation}
 1 \equiv v \pm b \beta^2 \equiv \alpha \beta+b \beta
 = (\alpha+b) \beta, \quad
 \alpha+b \equiv \beta \equiv 1 \pmod 2,
\end{equation}
 and thus
\begin{equation}
 \alpha+1 \equiv b\equiv b^2-4ac=
 D=(\alpha^2-4 \varepsilon)/\beta^2 \equiv \alpha^2-4 \varepsilon
 \equiv \alpha \pmod 2,
\end{equation}
 which is a contradiction.
Hence $v \pm b \beta^2$
 is even and the matrix (\ref{EQ_mat_g_epsilon})
 is an integer matrix.
Therefore
 $g$ acts on $A(L)$ as $\varepsilon \cdot \id$.
\end{proof}

%

Thanks to Proposition \ref{PROP_trace_g},
 the proof of Main Theorem is reduced to showing
\begin{align}
 A_\varepsilon :=& \{ \alpha \in \Z_{>0} \bigm|
 \disc(L)=-(\alpha^2-4\varepsilon)/\beta^2 ~
 (\exists \beta \in \Z_{>0}, ~ \exists L \in \Lprime)
 \} \\
 =&
 \begin{cases}
  \Z_{\geq 4} & \text{if\ \ } \varepsilon=1, \\
  \Z_{\geq 4} \setminus \{ 5,7,13,17 \}
   & \text{if\ \ } \varepsilon=-1.
 \end{cases}
\end{align}

\subsection{Symplectic case}

For $\alpha \in A_{+1}$, we have $\alpha \geq 3$.
If $\alpha=3$, then there exists $L\in \Lprime$
 such that $\disc(L)=-5$.
However,
 it follows from the table of indefinite binary quadratic forms
 \cite[Table 15.2]{CS} that
\begin{equation}
 L \cong \begin{pmatrix} 2& 1\\ 1&-2 \end{pmatrix}
 \not \in \Lprime,
\end{equation}
 which is a contradiction.
%
%
%
Hence $\alpha\geq 4$.

Conversely, for $\alpha \in \Z_{\geq 4}$, we set
\begin{equation}
 L:=\begin{pmatrix} 2& \alpha \\ \alpha & 2 \end{pmatrix}
 \in \mathbb{L}_{(1,1)}, \quad
 \disc(L)=-(\alpha^2-4).
\end{equation}
By \cite[Example 4]{GLP}, we have $L\in\Lprime$,
 and thus $\alpha \in A_{+1}$.

\subsection{Anti-symplectic case}

Now we consider the case $\varepsilon=-1$.
Let $\alpha \in \Z_{>0}$.

\subsubsection{Case of odd $\alpha$}
Suppose that $\alpha$ is odd.

First we assume that
$\alpha^2+4$ is not square-free, that is,
 $D:=\alpha^2+4=n^2 D_0$ with $n>1$.
Since $D\equiv 1 \pmod 4$, we have $D_0\equiv 1 \pmod 4$.
We define an even lattice $L$ by
\begin{equation}
 L:=\begin{pmatrix} 2n&n \\ n& -n(D_0-1)/2 \end{pmatrix},
 \quad
 \disc(L)=-n^2 D_0=-D.
\end{equation}
Since $D$ is not a square number, it follows that $L \in \Lprime$.
Hence $\alpha\in A_{-1}$ in this case.

Next we assume that $D:=\alpha^2+4$ is square-free.
Then $D$ is the discriminant of the real quadratic field
 $\Q(\sqrt{D})$.
By Lemma \ref{LEM_class_number},
 there exists $L\in \Lprime$ with $\disc(L)=-D$
 if and only if $h^+(D)>1$.
Therefore, in this case,
 the conditions
 $\alpha \in A_{-1}$ 
 and
 $\alpha \not\in \{ 1,3,5,7,13,17 \}$
 are equivalent
 by Theorem \ref{THM_Biro} (and Remark \ref{REM_negative_Pell}) below.


\begin{theorem}(\cite{Bi}) \label{THM_Biro}
For an odd integer $k>0$ such that $D:=k^2+4$ is square-free,
 the class number $h(D)$ of $\Q(\sqrt{D})$ is $1$
 if and only if $k \in \{ 1,3,5,7,13,17 \}$.
\end{theorem}

\begin{remark} \label{REM_negative_Pell}
In Theorem \ref{THM_Biro},
 since the equation $u^2-D v^2=-4$ has an integer solution
 (e.g.\ $(u,v)=(k,1)$), we have $h^+(D)=h(D)$.
\end{remark}


\subsubsection{Case of even $\alpha$}
Suppose that $\alpha$ is even and $\alpha \in A_{-1}$.
%
If $\alpha=2$, then there exists $L\in \Lprime$ such that
 $\disc(L)=-8$ or $-2$.
By \cite[Table 15.2]{CS}, we have
\begin{equation}
 L \cong \begin{pmatrix} -2& 0\\ 0&4 \end{pmatrix}
 \not\in \Lprime,
\end{equation} 
 which is a contradiction.
Therefore $\alpha\geq 4$.

Conversely, for even $\alpha\geq 4$,
 we define an even lattice $L$ by
\begin{equation}
 L:=\begin{pmatrix} \alpha&2 \\ 2& -\alpha \end{pmatrix}, \quad
 \disc(L)=-(\alpha^2+4).
\end{equation}
We show $L\in\Lprime$.
It is easy to check this for the case that
 $\alpha$ is divisible by $4$.
%
In the case $\alpha \equiv 2 \pmod 4$,
 we consider the following Gram matrix of $L$:
\begin{equation} 
 P^T\cdot
 \begin{pmatrix} \alpha&2 \\ 2& -\alpha \end{pmatrix}
 \cdot P =
 \begin{pmatrix} 4&0\\ 0& -(\alpha^2/4+1) \end{pmatrix}, \quad
 \text{where} \quad
 P=\begin{pmatrix} 1 & (\alpha-2)/4 \\ 1 & (\alpha+2)/4 \end{pmatrix}.
\end{equation}
We can apply Theorem \ref{Pell_eq} below
 for $\Delta=\alpha^2/4+1$, $a=2$ and $b=(\alpha^2/4+1)/2$
 to show $L\in\Lprime$.
Indeed, we have the continued fraction expansion
\begin{equation}
 \sqrt{\alpha^2/4+1}=[\alpha/2,\alpha,\alpha,\alpha,\ldots].
\end{equation}
The length of the period of this continued fraction
 is $1$.


\begin{theorem} \label{Pell_eq}(\cite{Mo})
Suppose that an integer $\Delta >2$ is not a square number.
Then the length of the period of the continued fraction expansion of
 $\sqrt{\Delta}$ is even if and only if one of the following holds.
\begin{enumerate}
\item There exists a factorization $\Delta=ab$ with $1< a < b$ such that
\begin{equation}
ax^2-by^2=\pm 1
\end{equation}
has an integer solution.
\item There exists a factorization $\Delta=ab$ with $1\leq a<b$ suth that
\begin{equation}
ax^2-by^2=\pm 2
\end{equation}
has an integer solution with $xy$ odd.
\end{enumerate}
\end{theorem}

%


\section{Applications}
\label{SECT_app}

In this section
 we discuss applications of our result.

First we give direct consequences of Main Theorem.
Let $\varphi$ be as in Main Theorem, that is,
 $\varphi$ is an automorphism
 of a projective K3 surface $X$ with Picard number $2$ such that
 $\ord(\varphi)=\infty$ and
 $\varphi^* \omega_X=\varepsilon \omega_X$
 ($\varepsilon \in \{ \pm 1 \})$.
Then $\tr(\varphi^*|_{S_X})=\alpha^2 - 2 \varepsilon$
 for some $\alpha \in A_{\varepsilon}$.
By \cite{O}, each fixed point of $\varphi$ on $X$ is isolated.
By the topological Lefschetz fixed point formula,
 the number of fixed points (with multiplicity) of $\varphi$
 is given as
\begin{equation} \label{EQ_Lefschetz_number}
 \sum_{i=0}^{4} (-1)^i \tr(\varphi^*|_{H^i(X,\C)})
 =1+(20 \varepsilon+\tr(\varphi^*|_{S_X}))+1=\alpha^2+18\varepsilon+2.
\end{equation}
As another consequence, we determine
 the spectral radius of $\varphi^*|_{H^2(X,\C)}$,
 which is defined as the maximum absolute value $|\lambda|$
 of eigenvalues $\lambda$ of $\varphi^*|_{H^2(X,\C)}$.
This plays an important role in the study of complex dynamics of K3 surfaces
 (see \cite{mcmullen1}).
In our case, the spectral radius is given as
\begin{equation}
 \frac{u+\sqrt{u^2-4}}{2}
 =
 \frac{\alpha^2 - 2 \varepsilon + \alpha \sqrt{\alpha^2 - 4\varepsilon}}{2},
\end{equation}
 where $u=\alpha^2 - 2 \varepsilon$.


Next we show that
 a fixed-point-free automorphism of a projective K3 surface
 with Picard number $2$ is nothing but
 the following Cayley--Oguiso automorphism:

\begin{theorem}(\cite{O}) \label{THM_Oguiso-example}
Any K3 surface $X$ with
\begin{equation} 
 S_X \cong \begin{pmatrix} 4& 2\\2&-4 \end{pmatrix}
\end{equation}
 admits a fixed-point-free automorphism $\varphi$
 of positive entropy with
\begin{equation}
 \varphi^*|_{S_X}=
 \begin{pmatrix} 5&8\\ 8&13 \end{pmatrix}, \quad
 \varphi^*|_{T_X}=-\id.
\end{equation}
\end{theorem}

\begin{remark}
In Theorem \ref{THM_Oguiso-example},
 $\Aut(X)$ is generated by $\varphi$
 \cite[Theorem 1.1]{Cayley-Oguiso-auto}.
\end{remark}



\begin{theorem} \label{THM_fixed-pint-free}
Any fixed-point-free automorphism $\varphi$
 of a projective K3 surface $X$ with Picard number $2$
 is of the form in Theorem \ref{THM_Oguiso-example}.
\end{theorem}
\begin{proof}
By \cite{O},
 we have $\tr(\varphi^*|_{S_X})=18$
 and $\varphi^* \omega_X=-\omega_X$.
%
Hence
 $\disc(S_X)=-20$ or $-5$
 by Proposition \ref{PROP_key_prop}.
From \cite[Table 15.2]{CS},
 we find that a lattice belonging to $\mathbb{L}_{(1,1)}$
 with discriminant $-20$ or $-5$
 is isomorphic to one of the following:
\begin{equation} 
 \begin{pmatrix} 4&2\\ 2&-4 \end{pmatrix},
 \begin{pmatrix} -2&0\\ 0& 10 \end{pmatrix},
 \begin{pmatrix} 2& 1 \\ 1 &-2 \end{pmatrix}.
\end{equation}
Since $S_X\in\Lprime$ by Proposition \ref{PROP_inf_S_X},
 we have
\begin{equation}
 S_X \cong
 \begin{pmatrix} 4&2\\ 2&-4 \end{pmatrix}, \quad
 \varphi^*|_{S_X} = U \text{~or~} U^{-1}, \quad
 U:=
 \begin{pmatrix} 5&8\\ 8&13 \end{pmatrix}
\end{equation}
 by Proposition \ref{PROP_key_prop}.
Here $U$ and $U^{-1}$ correspond to
 $(u,v)=(18,\pm 4)$ satisfying
 $u^2-20 v^2=4$.
Since $O^+(S_X)$ is isomorphic to the infinite dihedral group,
 $U$ and $U^{-1}$ are conjugate in $O^+(S_X)$.
In fact, we have
\begin{equation}
 V:=\begin{pmatrix} 1&0\\ 1&-1 \end{pmatrix} \in O^+(S_X), \quad
 V^{-1} \cdot U^{-1} \cdot V=U.
\end{equation}
Hence we may assume $\varphi^*|_{S_X}=U$
 by changing basis of $S_X$ if necessary.
We have $\varphi^*|_{T_X}=-\id$
 by Proposition \ref{PROP_inf_S_X}.
\end{proof}

\begin{remark} \label{REM_uniqueness-cayley-oguiso}
One can show that the action of a Cayley--Oguiso automorphism
 (as in Theorem \ref{THM_Oguiso-example})
 on the K3 lattice is essentially unique
 by applying Nikulin's lattice theory \cite{nikulin79int}.
Combined with Theorem \ref{THM_fixed-pint-free},
 this implies that a pair $(X,\varphi)$ of a K3 surface $X$
 with Picard number $2$
 and a fixed-point-free automorphism $\varphi \in \Aut(X)$
 is unique up to equivariant deformation.
Moreover, $X$ is realized as a determinantal quartic surface
 and $\varphi$ is constructed by using cofactor matrix.
See \cite{Cayley-Oguiso-auto} for details.
%
%
\end{remark}

\bigskip
Kenji Hashimoto

Graduate School of Mathematical Sciences,

The University of Tokyo,

3-8-1 Komaba, Maguro-ku, Tokyo, 153-8914, Japan

hashi@ms.u-tokyo.ac.jp

\bigskip
JongHae Keum

School of Mathematics,

 Korea Institute for Advanced Study,

 85 Hoegiro, Dongdaemun-gu, Seoul 02455, Korea

jhkeum@kias.re.kr

\bigskip
Kwangwoo Lee

Division of Liberal Arts and Science,

Gwangju Institute of Science and Technology,

123 Cheomdangwagi-ro, Buk-gu, Gwangju 61005, Korea

kwlee@gist.ac.kr

\begin{thebibliography}{KSWZxx}
\bibitem[BHPV]{BHPV}
W. P. Barth, K. Hulek, C. A. M. Peters and A. van de Ven, Compact complex surfaces, Second edition, Springer-Verlag, Berlin, 2004.
\bibitem[Bi]{Bi}
A. Biro, Yokoi's conjecture, Acta Arithmetica, {\bf106.1} (2003), 85--104.
\bibitem[BR]{burnsrapoport75}
D. Burns and M. Rapoport, On the Torelli problem for k\"ahlerian K-$3$ surfaces, Ann.\ Sci.\ Ecole Norm.\ Sup.\ (4) 8 (1975), no.\ 2, 235--273.
\bibitem[CS]{CS}J. H. Conway and N. J. Sloane, Sphere packings, lattices and groups, Third edition, Springer-Verlag, New York, 1999.
\bibitem[D]{D}
L. E. Dickson, Introduction to the theory of numbers, Dover, 1954.
\bibitem[FGvGvL]{Cayley-Oguiso-auto}
D. Festi, A. Garbagnati, B. van Geemen and R. van Luijk,
The Cayley--Oguiso automorphism of positive entropy on a K3 surface,
J. Mod. Dyn. 7 (2013), no. 1, 75--97.
\bibitem[FT]{FT-number-theory}
A.\ Fr\"ohlich and M. J. Taylor,
Algebraic number theory,
Cambridge University Press, 1993.
\bibitem[GLP]{GLP}
F. Galluzzi, G. Lombardo, C. Peters, Automorphs of indefinite binary quadratic forms and K3-surfaces with Picard number 2, Rend. Sem. Mat. Univ. Politec. Torino, {\bf 68(1)} (2010), 57--77.
\bibitem[J]{J}
B. W. Jones, The arithmetic theory of quadratic forms. Carcus Monograph Series, no. 10. The mathematical Association of America, Buffalo, N.Y., 1950.
\bibitem[M]{mcmullen1}
C. T. McMullen, Dynamics on K3 surfaces: Salem numbers and Siegel disks, J. Reine Angew. Math., {\bf 545} (2002), 201--233.
\bibitem[Mo]{Mo}
R. A. Mollin, A continued fraction approach to the Diophantine equation $ax^2-by^2=\pm 1$, JP Journal of Algebra, Number Theory and Applications, {\bf 4} (2004), 159--207.
\bibitem[N1]{nikulin79fin}
 V. V. Nikuiln,
 Finite groups of K\"ahlerian surfaces of type K3 (English translation),
 Trans.\ Moscow Math.\ Soc.\ 38 (1980), 71--137.
\bibitem[N2]{nikulin79int}
 V. V. Nikulin,
 Integral Symmetric bilinear forms and some of their applications
  (English translation),
 Math.\ USSR Izv.\ 14 (1980), 103--167.
\bibitem[O]{O}
K. Oguiso, Free Automorphisms of positive entropy on smooth K$\ddot{a}$hler surfaces, arXiv:1202.2637, to appear in Adv. Stud. Pure Math.
\bibitem[PS]{PS}I. I. Pjatecki\v{i}-\v{S}apiro and  I. R. \v{S}hafarevi\v{c}, Torelli's theorem for algebraic surfaces of type K3 (in Russian), Izv. Akad. Nauk SSSR Ser. Mat. 35 (1971), 530-572.
\bibitem[T]{Todorov}A. Todorov, Applications of the K\"{a}hler--Einstein--Calabi--Yau metric to moduli of K3 surfaces, Invent. Math., 61 (1980), 251--265.
\end{thebibliography}
\end{document}